\documentclass[a4paper,11pt]{amsart}

\usepackage[foot]{amsaddr}
\usepackage{thm-restate}
\usepackage{amsthm}
\usepackage{amsmath,amsfonts,amssymb}
\usepackage{thmtools, thm-restate}
\usepackage{parskip}
\usepackage[left = 3.5cm, right = 3.5cm, top = 3.5cm, bottom = 3.5cm]{geometry}
\usepackage{enumitem}

\date{31st March 2021}

\newcommand{\enk}{\ensuremath{\mathcal{E}_{\enklet{},k}}}
\newcommand{\enklet}{m}

\renewcommand{\leq}{\leqslant}
\renewcommand{\geq}{\geqslant}
\newcommand{\abs}[1]{\ensuremath{\lvert{} #1 \rvert{}}}

\newcommand{\floor}[1]{\ensuremath{\lfloor{} #1 \rfloor{}}}
\newcommand{\ceil}[1]{\ensuremath{\lceil{} #1 \rceil{}}}

\newtheorem{theorem}{Theorem}[section]
\newtheorem{lemma}[theorem]{Lemma}

\theoremstyle{definition}
\newtheorem{dfn}[theorem]{Definition}
\binoppenalty=\maxdimen
\relpenalty=\maxdimen
\title{On the extremal function for graph minors}

\author{Andrew Thomason$^1$}
\author{Matthew Wales$^1$}
\address{$^1$DPMMS, University of Cambridge, CB3 0WB, United Kingdom}

\begin{document}

\begin{abstract}
  For a graph $H$, let
  $c(H)=\inf\{c\,:\,e(G)\geq c|G| \mbox{ implies } G\succ H\,\}$, where
  $G\succ H$ means that $H$ is a minor of $G$. We show that if $H$ has
  average degree $d$, then
$$
c(H)\le (0.319\ldots+o_d(1))|H|\sqrt{\log d}
$$
where $0.319\ldots$ is an explicitly defined constant.
This bound matches a corresponding lower bound shown to hold for almost all
such $H$ by Norin, Reed, Wood and the first author.
\end{abstract}

\maketitle

\section{Introduction}

A graph $H$ is a \textsl{minor} of $G$, written $G\succ H$, if there exist
non-empty disjoint subsets $U_v\subset V(G)$, $v\in V(H)$, such that each
$G[U_v]$ is connected, and there is an edge in $G$ between $U_v$ and $U_w$
whenever $vw$ is an edge in~$H$. Thus $H$ can be obtained from $G$ by a
sequence of edge contractions and deletions and vertex deletions.
	
It is natural to ask, for a given $H$, what conditions on a graph $G$
guarantee that it contains $H$ as a minor. Mader~\cite{Mader1967} proved
the existence of the extremal function~$c(H)$, defined as follows.

\begin{dfn}
  For a graph $H$, let
  $c(H) = \inf \{\, c\,:\, e(G)\geq c|G| \mbox{ implies } G\succ H\,\}$.
\end{dfn}

Mader~\cite{Mader68} later proved the bound $c(K_t)\le 8t\log_2t$ for the
complete graph $H=K_t$. Bollob\'as, Catlin and Erd\H{o}s~\cite{BollCatErd}
realised that random graphs $G=G(n,p)$ give a good lower bound for
$c(K_t)$. Indeed, by choosing $n$ and $p$ suitably, one obtains
$c(K_t)\ge(\alpha+o(1))t\sqrt{\log t}$, where the constant $\alpha$ is
described here.
\begin{dfn} The constant $\alpha$ is given by
$$
\alpha\,=\,\max_{0<p<1} {p/2\over\sqrt{\log(1/(1-p))}}\,=\,0.319\ldots\,,
$$
with $p=0.715\ldots$
giving the maximum value.
\end{dfn}

Kostochka~\cite{Kostochka1,Kostochka2} (see also~\cite{thom84}) proved that
$c(K_t)$ is in fact of order $t\sqrt{\log t}$. Finally, it was
shown~\cite{thom01} that $c(K_t)=(\alpha+o(1))t\sqrt{\log t}$.

Myers and Thomason~\cite{myersthom} considered general graphs $H$ with $t$
vertices and at least $t^{1+\tau}$ edges.  They defined a graph parameter
$\gamma(H)$ and proved that \\$c(H)=(\alpha\gamma(H)+o(1)) t\sqrt{\log
  t}$. We say more about $\gamma(H)$ in~\S\ref{secfuture}, but here it is
enough to say that $\gamma(H)\leq\sqrt{\tau}$, and in fact
$c(H)=(\alpha\sqrt{\tau}+o(1)) t\sqrt{\log t}$ for almost all $H$ with
$t^{1+\tau}$ edges, and for all regular $H$ of this kind.
Myers~\cite{Myers11} further showed that the extremal graphs are all
essentially disjoint unions of pseudo-random graphs (in the sense
of~\cite{thomasonpseudo}) having the same order and density as the random
graphs discussed above. Thus $c(H)$ is determined very precisely when
$\tau$ is bounded away from zero, but these results give little useful
information when $\tau=o(1)$.

Reed and Wood~\cite{reedwood,reedwoodcorrig} realised that it is better to
express $c(H)$ in terms of the average degree $d$ of~$H$. For example, the
results just mentioned imply that $c(H)\le(\alpha+o(1)) |H|\sqrt{\log d}$
if $\log d \ne o(\log |H|)$, with equality in many cases. Reed and Wood
showed that $c(H)\le 1.9475|H|\sqrt{\log d}$ holds for all~$H$, provided
$d$ is large.

The actual behaviour of $c(H)$ can be qualitatively different, though, when
$\log d = o(\log |H|)$, because random graphs themselves cannot serve as
extremal graphs. In fact, Alon and F\"uredi~\cite{AlonFuredi} showed that,
if the maximum degree of $H$ is at most $\log_2 |H|$, then the random graph
$G(|H|,p)$ almost surely contains $H$ as a {\em spanning} subgraph when
$p>1/2$. Indeed, $c(H)$ can be much smaller than $|H|\sqrt{\log d}$, even
if $H$ is regular: Hendrey, Norin and Wood~\cite{HNW} have shown that
$c(H)=O(|H|)$ when $H$ is a hypercube.

But this kind of behaviour turns out to be rare. Norin, Reed, Thomason and
Wood~\cite{NRTW} recently found a different class of graphs that can serve
as extremal graphs. These are blowups of small random graphs but are not
themselves random (though they are pseudo-random). Their method showed that
$c(H)\ge (\alpha +o_d(1))|H|\sqrt{\log d}$ holds for almost all $H$ of
average degree~$d$. (More exactly, given $\epsilon>0$ and
$d>d_0(\epsilon)$, then for each $t>d$, when $H$ is chosen at random with
$t$ vertices and average degree~$d$, we have
$\Pr[c(H)> (\alpha -\epsilon)t\sqrt{\log d}] > 1-\epsilon$.)  The bound on
$c(H)$ was conjectured to be tight; that is,
$c(H)= (\alpha +o_d(1))|H|\sqrt{\log d}$ almost always.

Our main purpose here is to prove the next theorem, which strengthens the
result of Reed and Wood~\cite{reedwood} and, in particular, settles the
aformentioned conjecture positively.

\begin{theorem}
\label{T:key}
  Given $\epsilon>0$, there exists $D(\epsilon)$, such that if
  $d>D(\epsilon)$ then every graph $H$ of average degree $d$ satisfies
  $c(H) < (\alpha+\epsilon)|H|\sqrt{\log d}$.
\end{theorem}

The proof follows very broadly the strategy of~\cite{thom01}, used also
in~\cite{myersthom}. It splits into two cases. The first is where $G$ has
density bounded away from zero, and is reasonably connected; this is
addressed in~\S\ref{secdense}. The second case is where $G$ is itself sparse
but still has dense vertex neighbourhoods as well as reasonable
connectivity; this is dealt with in~\S\ref{secsparse}. Nevertheless the
methods of~\cite{thom01,myersthom} are not adequate to handle the situation
where $|H|$ is much bigger than~$d$, and new ideas are needed. These are
described in the appropriate sections.

Our notation is more or less standard. For example, given a graph~$H$, then
(as used above) $|H|$ denotes the number of vertices, $e(H)$ the number of
edges, and $\delta(H)$ the minimum degree. If $X$ is a subset of the vertex
set $V(H)$ then $\Gamma(X)$ denotes the set of vertices not in $X$ that
have a neighbour in~$X$. The subgraph of $H$ induced by $X$ is denoted
by~$H[X]$.

The proof of Theorem~\ref{T:key} begins with the following
families of graphs.

\begin{dfn}
  Let $\enklet{}> 1$ and $ 0\leq k\leq \enklet{}/2$ be real numbers. Define
  \begin{align*} \enk{} = \{\,G\,:\, \abs{G}\geq \enklet{}\,,\,
    e(G)>\enklet{}\abs{G}-\enklet{}k\,\}\,.
  \end{align*}
\end{dfn}

The main usefulness of the class $\enk{}$ to the study of $c(H)$,
demonstrated by Mader~\cite{Mader68}. is that a minor-minimal element of
$\enk{}$ (that is, a graph $G\in\enk{}$ which has no proper minor in
$\enk{}$) enjoys the properties set out in the next lemma.

\begin{lemma}
  \label{L:enkprop}
  Let $G$ be a minor-minimal element of $\mathcal{E}_{\enklet{},k}$. Then
  $|G|\geq \enklet{}+1$, $e(G)\leq \enklet{}|G|-\enklet{}k+1$,
  $\enklet{}<\delta(G)<2\enklet{}$, $\kappa(G)>k$, and every edge of $G$ is
  in more than $\enklet{}-1$ triangles.
\end{lemma}
\begin{proof}
  The proof is standard and elementary (see for
  example~\cite[Section~2]{thom01}), though usually $\enklet{}$ and $k$ are
  taken to be integers, so we provide a brief sketch. There are no graphs
  $G\in\enk{}$ with $\enklet{}\leq |G| < \enklet{}+1$ because in this range
  ${|G|\choose 2} < \enklet{}|G|-\enklet{}k$. Hence the removal of a vertex
  of $G$, or the contraction or removal of an edge, violates the size
  condition, which yields all the claimed properties except
  $\kappa(G)>k$. To obtain this, consider a cutset $S$ and a component $W$
  of $G-S$. The condition $\delta(G)>\enklet{}$ implies that both the
  minors $G[W\cup S]$ and $G\setminus W$ have more than $\enklet{}$
  vertices, and hence
  \begin{align*}
  e(G) \leq e(G[W\cup S])+e(G\setminus W)\le
  \enklet(|W|+|S|)-\enklet{}k+\enklet{}(|G|-|W|)-\enklet{} k\end{align*}
  yielding $|S|>k$ as claimed.
\end{proof}

We now state the main theorems of Sections 2 and 3 respectively, and show
how they imply Theorem \ref{T:key} .

\begin{restatable}[]{theorem}{sectwothm}
\label{T:sectwoweak}
Given $\epsilon >0$, there exists $D_1(\epsilon)$, such that
if $H$ is a graph of average degree $d>D_1$, and $G$ is a graph of density at
least $p+\epsilon$, with the properties $\epsilon < p <1-\epsilon$,
$|G|\geq |H|\sqrt{\log_{1/(1-p)}d}$ and $\kappa(G)\geq \epsilon |G|$, then
$G\succ H$.
\end{restatable}

\begin{restatable}[]{theorem}{secthreethm}
\label{T:secthree}
Given $\epsilon >0$, there exists $D_2(\epsilon)$, such that, if $H$ is a
graph of average degree $d>D_2$, $\enklet{}$ is a number with
$\enklet{}\geq \epsilon|H|\sqrt{\log d}$, and $G$ is a graph satisfying
$|G|\geq D_2\enklet{}$, $\kappa(G)\geq D_2|H|$, $e(G)\leq \enklet{}|G|$ and
every edge of $G$ lies in more than $\enklet{}-1$ triangles, then
$G\succ H$.
\end{restatable}

\begin{proof}[Proof of Theorem \ref{T:key}]
  We may assume that $0<\epsilon <1/4$. Let $H$ be a graph of average
  degree~$d$, let $\enklet{} = (\alpha+\epsilon)|H|\sqrt{\log d}$ and let
  $k = \epsilon \enklet{}/2$. We need to show that $G\succ H$ provided
  $e(G)\ge \enklet{}|G|$ and $D(\epsilon)$ is large. Now $G\in\enk$, and
  (replacing $G$ by a minor of itself if necessary) from now on we assume
  that $G$ is minor minimal in~$\enk$ (note that we thereby forego the
  inequality $e(G)\ge \enklet{}|G|$ but we still have \\$e(G)>
  \enklet{}|G|-\enklet{}k$). Thus $G$
  has all the properties stated in Lemma~\ref{L:enkprop}. We may assume
  that $D(\epsilon)$
  is large enough that $\epsilon(\alpha+\epsilon)\sqrt{\log d}/2 > D_2$, where
  $D_2=D_2(\epsilon)$ is the constant of Theorem~\ref{T:secthree}, and
  therefore $\kappa(G)> k> D_2|H|$. Thus if $\abs{G}\geq
  D_2\enklet{}$, then $G$ satisfies the conditions of
  Theorem~\ref{T:secthree} (provided $D(\epsilon)>D_2$) and $G\succ
  H$. 
  So suppose instead that $\abs{G}<
  D_2\enklet{}$. Then $e(G)>\enklet{}|G|-\enklet{}k \geq
  \enklet{}|G|(1-\epsilon/2)$, so $G$ has density at least
  $2\enklet{}(1-\epsilon/2)/(|G|-1)>1/D_2$.  Let $\epsilon'=\epsilon/2D_2$,
  and let $p+\epsilon'$ be the density of~$G$, so $\epsilon' < p <
  1-\epsilon'$. Observe that $p>1/2D_2$ so $\epsilon p>\epsilon'$. Therefore
  $p(1+\epsilon)> p+\epsilon' >2\enklet{}(1-\epsilon/2)/(|G|-1)$, so we have
  $|G|> (2\enklet{}/p)(1-\epsilon/2)/(1+\epsilon) >(2\alpha/p)|H|\sqrt{\log
    d}$.
  Now $2\alpha/p \geq
  1/\sqrt{\log(1/(1-p))}$ holds by the definition of~$\alpha$, so $|G|\geq
  |H|\sqrt{\log_{1/(1-p)}d}$. As for the connectivity of~$G$, we have
  $\kappa(G)\geq k=\epsilon \enklet{}/2 > \epsilon
  |G|/2D_2=\epsilon'|G|$. We
  may now apply Theorem~\ref{T:sectwoweak} to $G$ with $\epsilon'$ in place
  of~$\epsilon$, to see that $G\succ H$, provided $D(\epsilon)>D_1(\epsilon')$.
\end{proof}

\section{The Dense Case}\label{secdense}
In this section, our main aim will be to prove
Theorem~\ref{T:sectwoweak}. In fact, it will turn out to be useful to prove
a slightly stronger version of the theorem, namely
Theorem~\ref{T:sectwostrong}, in which $H$ is a {\em rooted} minor, which
is to say we specify, for each $v\in V(H)$, a vertex of $G$ that must lie
in the class~$U_v$. This will be needed when we come to the sparse case in
the next section.

The essence of the proof is to choose the parts $U_v$ at random
from~$G$. It is very unlikely that the parts so chosen will be connected;
to get round this, we first put aside a few vertices of $G$ and choose the
$U_v$ from the remainder, using the put aside vertices afterwards to
augment the sets $U_v$ into connected subgraphs. In this way, all that we
require of the random sets $U_v$ is that there is an edge in $G$ between
$U_v$ and $U_w$ whenever $vw\in E(H)$. This procedure nearly works, but it
throws up a few ``bad'' parts $U_v$ that cannot be used, and even among the
good parts there will be a few edges $vw\in E(H)$ for which there is no
$U_v$--$U_w$ edge. In~\cite{thom01}, where $H=K_t$, this was not a big
problem: the initial aim is changed to finding instead a $K_{(1+\beta)t}$
minor, at no real extra cost if $\beta$ is small, and the few blemishes in
this minor still leave us with a $K_t$ minor. In~\cite{myersthom}, where
$H$ has average degree at least $t^\epsilon$, a similar solution is found;
a $H+K_{\beta t}$ minor is aimed for, which even with up to $\beta t$
blemishes still leaves an $H$ minor. The method works in~\cite{myersthom}
because if $\beta(\epsilon)$ is small then $\gamma(H)$ and
$\gamma(H+K_{\beta t})$ are relatively close, as are $c(H)$ and
$c(H+K_{\beta t})$.

This method fails completely for sparse~$H$, because $c(H)$ and
$c(H+K_{\beta t})$ are far from each other, so we need a new approach. We
randomly partition $G$ (after setting aside some vertices) into somewhat
more than $|H|$ parts but without predetermining which part is assigned to
which vertex of $H$. After discarding the few bad parts we still have $|H|$
good parts left, each of which has an edge to most of the other good
parts. The good parts are now randomly assigned to the vertices of~$H$; it
turns out that this is enough to ensure that not too many (fewer than
$|H|$) edges $vw\in E(H)$ are left with no $U_v$--$U_w$ edge. For
these few missing edges, we can find $U_v$--$U_w$ paths at the final stage
when we make all the $U_v$ connected. In this way we obtain the required
$H$ minor.

\subsection{Almost-$H$-compatible equipartitions}\label{subalmost}

\begin{dfn}
  An \textit{equipartition} of $G$ is a partition of $V(G)$ into parts
  $V_i$ whose sizes differ by at most one. A
  $j$\textit{-almost-}$H$\textit{-compatible equipartition} of $G$ is an
  equipartition into parts $V_v$, $v\in V(H)$, where there are at most $j$
  edges $vw$ of $H$ for which there is no edge between $V_v$ and $V_w$. An
  $H$\textit{-compatible equipartition} of $G$ is a 0-almost-$H$-compatible
  equipartition.
\end{dfn}

We now give the details of the argument sketched above.  Here,
in~\S\ref{subalmost}, we find an almost-$H$-compatible equipartition in a
dense graph, and then in~\S\ref{subconnproj} we show how to connect up the
parts of the equipartition, as well as `adding in' the missing edges.

The result of taking a random equipartition is described by the next
lemma. We remark that, in the proof, the parts are not chosen entirely
randomly, but subject to the constraint that each part gets its fair share
of high and low degree vertices --- this helps to control the number of
``bad'' parts. The lemma is more or less identical
to~\cite[Theorem~3.1]{thom01}, and we keep its technical form so that we
can copy it over with very little comment.

\begin{lemma}
\label{T:generalpartition}
Let $G$ be a graph of density at least $p$, let $l\geq 2$ be an integer,
and let $s = \lfloor |G|/l\rfloor\geq 2$. Let $\omega > 1$ and
$0<\eta < p$ . Then $G$ has an equipartition into at least
$s-\frac{2s}{\omega \eta}$ parts, with at most
$2s^2(6\omega)^l(\frac{1-p}{1-\eta})^{(1-\eta)l(l-1)}$ pairs of parts having
no edge between them.
\end{lemma}
\begin{proof}
  The argument is essentially exactly that of~\cite[Theorem~3.1]{thom01},
  though the conclusion stated there is very slightly different.  The
  argument involves removing $|G|-sl$ vertices, and then choosing randomly,
  from a certain distribution, an equipartition of the remaining vertices
  into $s$ parts of size~$l$. Some of these parts are of no use (they are
  called ``unacceptable'' in~\cite{thom01} and ``bad'' above). Amongst the
  other parts, some pairs will be defective in that they fail to have an
  edge between them. At the start of the final paragraph of the proof
  in~\cite{thom01}, it is stated that there is a partition with at most
  $2s/\omega\eta$ unacceptable parts and at most
  $2s^2(6\omega)^l(\frac{1-p}{1-\eta})^{(1-\eta)l(l-1)}$ defective pairs
  amongst the other parts. For the conclusion of
  \cite[Theorem~3.1]{thom01}, the unacceptable parts and one part from each
  defective pair are thrown away. For the conclusion of the present lemma,
  we keep all the acceptable parts. We then take the vertices from the
  unacceptable parts, together with the $|G|-sl$ vertices initially
  removed, and redistribute them amongst the acceptable parts so as to
  obtain the desired equipartition.
\end{proof}

\begin{lemma}\label{L:logprop}
  Let $\epsilon\leq1/2$ and let $0<x<1-\epsilon$. Then
  $\sqrt{\frac{\log(x+\epsilon)}{\log x}} \leq 1-\epsilon$.
\end{lemma}
\begin{proof}
This is \cite[Lemma~3.1]{myersthom}, except that there the condition is
$0<\epsilon\leq x\leq 1-\epsilon$. However, though the implied condition
$\epsilon \leq1/2$ is used in the proof, the condition $\epsilon\leq x$ is
not, and the proof works for $0<x$.
\end{proof}

Here is the main result of this subsection.

\begin{theorem}
\label{C:almostcompatible}
Given $\epsilon >0$, there exists $D_3(\epsilon)$, such that if $H$ is a
graph of average degree $d>D_3$, and $G$ is a graph of density at least
$p+\epsilon$, with the properties $\epsilon < p <1-\epsilon$ and
$|G|\geq |H|\sqrt{\log_{1/(1-p)}d}$, then $G$ contains an
$|H|d^{-\epsilon/3}$-almost-$H$-compatible equipartition.
\end{theorem}
\begin{proof} 
  We shall apply Lemma~\ref{T:generalpartition} to $G$, replacing $p$ in
  the lemma by $p+\epsilon$ (as we may), and using suitably chosen
  parameters $l$, $k$, $\eta$, $\omega$ so that $|H|\leq s(1-2/\omega\eta)$
  and $2/\omega\eta\leq 1/10$.  Suppose we have done this, obtaining an
  equipartition into $k\geq s(1-2/\omega\eta)$ parts. Note that
  $\epsilon<1/2$, and that we can make $l$, $|H|$ and $s$ as large as we
  like by making $D_3(\epsilon)$ large.  In particular
  $2s^2/{k\choose2}\leq 5$. Thus, writing $P$ for the probability that a
  randomly chosen pair of parts has no edge between, we have
  $P\le 5(6\omega)^l(\frac{1-p-\epsilon}{1-\eta})^{(1-\eta)l(l-1)}$. Now
  randomly label $|H|$ of the parts as $U_v$, $v\in V(H)$. The expected
  number of edges $vw$ of $H$ without a $U_v$--$U_w$ edge is
  $Pe(H)=P|H|d/2$.  We shall choose the parameters so that
  $P\le d^{-1-\epsilon/3}$, and therefore there is some labelling with
  fewer than $|H|d^{-\epsilon/3}$ such edges. Redistributing the vertices
  of the non-labelled parts amongst the $|H|$ labelled parts then gives the
  desired almost-$H$-compatible equipartition, so proving the theorem.

  All that remains is to choose suitable parameters. To start with, let
  \\$l=\ceil{(1-\epsilon/4)
    \sqrt{\log_{1/(1-p)}d}}$ and $s=\floor{|G|/l}$.
  We take $\omega=20/\epsilon^2\eta$, with $\eta$ still to be
  chosen. Certainly $2/\omega\eta\leq 1/10$, as needed, and moreover
  $s(1-2/\omega\eta)\geq |H|$. 

  We turn now to the bound on $P$, which is
  $P\le 5(6\omega)^l(\frac{1-p-\epsilon}{1-\eta})^{(1-\eta)l(l-1)}$.
  Choose $\eta$ small so that $\eta<\epsilon/4$ and
  $(1-2\epsilon)^{\epsilon/2}<(1-\eta)^{1-2\eta}$. Since $p>\epsilon$ this
  implies 
  $$ \left(\frac{1-p-\epsilon}{1-\eta}\right)^{(1-2\eta)}<
  (1-p-\epsilon)^{1-\epsilon}
  \frac{(1-2\epsilon)^{\epsilon-2\eta}}{(1-\eta)^{(1-2\eta)}}
  < (1-p-\epsilon)^{1-\epsilon}\,.$$
  Make $l$ large enough so that
  $(1-2\eta)l^2<(1-\eta)l(l-1)$. Then \\$P\leq 5(6\omega)^l
  (1-p-\epsilon)^{(1-\epsilon)l^2}$.
  By Lemma~\ref{L:logprop} we have
  $\frac{\log(1-p)}{\log(1-p-\epsilon)} \leq (1-\epsilon)^2$, meaning that
  $(1-p-\epsilon)^{(1-\epsilon)^2}\leq 1-p$, so
  $P\leq5(6\omega)^l(1-p)^{l^2/(1-\epsilon)}$. Since\\
  $l=\ceil{(1-\epsilon/4)
    \sqrt{\log_{1/(1-p)}d}}$ and $(1-\epsilon/4)^2/(1-\epsilon) > 1+\epsilon/2$, we obtain\\
  $P\leq5(6\omega)^l d^{-1-\epsilon/2}$. Finally, by making $d$ large, we
  have $P\leq d^{-1-\epsilon/3}$, as desired.
\end{proof}

\subsection{The connector and the projector}\label{subconnproj}

As mentioned at the start of~\S\ref{secdense}, when proving
Theorem~\ref{T:sectwostrong} we first put aside a small set for later
use. This set is actually made up of two special sets that we call the {\em
  connector} and the {\em projector}. The connector will contain many short
paths between all pairs of vertices, and the projector will allow us to
connect many sets $X$ by connecting only about $\log\abs{X}$
vertices of the projector.

We borrow a couple of very straightforward lemmas from~\cite{thom01}.

\begin{lemma}[\protect{\cite[Lemma~4.2]{thom01}}]
\label{L:connectivity}
Let $G$ be a graph of connectivity $\kappa>0$ and let $u,v\in V(G)$. Then
$u$ and $v$ are joined in $G$ by at least $\kappa^2/4|G|$ internally
disjoint paths of length at most $2|G|/\kappa$.
\end{lemma}

\begin{lemma}[\protect{\cite[Lemma~4.1]{thom01}}]
\label{L:bipartiteprojection}
Let $G$ be a bipartite graph on vertex classes $A,B$ with the property
every vertex in $A$ has at least $\gamma \abs{B}$ neighbours in $B$, where
$\gamma>0$. Then there is a set $M\subseteq B$ with
$|M|\leq\floor{\log_{1/(1-\gamma)}\abs{A}}+1$ such that every vertex in $A$
has a neighbour in $M$, that is, $A\subset\Gamma(M)$.
\end{lemma}

The next theorem provides us with a connector and a projector. The form of
the theorem allows us to put aside  not just these two sets but also a third
set $R$ which will form the roots of our rooted $H$ minor.

\begin{theorem} \label{T:connectorprojector} Given $\eta>0$, there exists
  $D_4(\eta)$ such that, if $G$ is a graph with $|G|\geq D_4$ and
  $\kappa(G)\geq8\eta|G|$, then for each set $R\subset V(G)$ with
  $|R|\leq\eta |G|$, there exist subsets $C$ (the connector) and $P$ (the
  projector) in $V(G)$, disjoint from each other and from $R$, with the
  following properties:
  \begin{enumerate}
  \item{}$\abs{C},\abs{P}\leq 2\eta |G|$,
  \item{} each pair of vertices $u,v \notin C$ is joined by
    at least $2\eta^{2+1/2\eta}|G|$ internally disjoint
    paths of length at most $1/2\eta$ whose internal vertices lie in~$C$,
  \item{} each vertex not in $P$ has at least $2\eta^2|G|$ neighbours
    in $P$, and
  \item{} for every $Y\subset P$ with $|Y|\leq\eta^2|G|$ and 
    for every $X\subset V(G)-P-C$, there is a set $M$ in $P-Y$ with
    $X\subset \Gamma(M)$ and
    $|M|\leq\floor{\log_{1/(1-\eta/2)}\abs{X}}+1$.
  \end{enumerate}
\end{theorem}
\begin{proof}
  First, we construct $C$. We put vertices of $G-R$ inside $C$
  independently at random with probability $\eta$. By Markov's Inequality,
  $\abs{C}\leq 2\eta |G|$ holds with probability at least $1/2$.

  Define $\delta$ by $\kappa(G)=\delta|G|$, so $\delta\geq 8\eta$.  Let
  $u,v\notin C$. Applying Lemma~\ref{L:connectivity} to the graph
  $G-(R\setminus\{u,v\})$, noting that its connectivity is at least
  $(\delta-\eta)|G|$, we see that $u$ and $v$ are joined by at least
  $(\delta-\eta)^2|G|/4\geq \delta^2|G|/16$ vertex-disjoint paths of length
  at most $2/(\delta-\eta)\leq 4/\delta$, whose internal vertices are not
  in~$R$.  The probability that a given one of these paths has all internal
  vertices inside $C$ is at least $\eta^{4/\delta}$. We therefore expect at
  least $\delta^2\eta^{4/\delta}|G|/16$ vertex disjoint paths of length at
  most $4/\delta$ with all internal vertices inside~$C$. The paths are
  disjoint so the probabilities for different paths are independent of each
  other. Hence by a standard Chernoff bound (see for example
  \cite{AlonSpencer}), except with probability at most
  $|G|^2\exp(-\delta^2\eta^{4/\delta}|G|/128)$, all pairs of vertices are
  joined by at least $\delta^2\eta^{4/\delta}|G|/32$ vertex disjoint paths
  of length at most $4/\delta$ whose internal vertices lie inside~$C$. By
  making $D_4$ large we can ensure this probability is less than $1/2$.
  Combining this with the observation in the previous paragraph, we see
  that with positive probability there is a set $C$ satisfying
  $\abs{C}\leq 2\eta |G|$ and such that every pair $u,v \notin C$ is joined
  by at least $\delta^2\eta^{4/\delta}|G|/32\geq 2\eta^{2+1/2\eta}|G|$
  paths inside $C$ of length at most $4/\delta\leq 1/2\eta$. 
  Fix now such a choice of $C$; then $C$ satisfies properties~(1) and~(2).
	
  We next construct $P$. Consider $G-C-R$. Since $G$ has minimum degree at
  least $\kappa(G)=\delta |G|$, every vertex has at least
  $\delta|G|-4\eta|G|\geq\delta|G|/2$ neighbours inside $G-C-R$. Place
  vertices in $P$ independently at random with probability $\eta$. As
  before, with probability at least $1/2$, $\abs{P}\leq 2\eta|G|$.
  Given a vertex not in $P$, the number of its neighbours in $P$ is
  binomially distributed with mean at least $\eta\delta|G|/2$. Again, by a
  Chernoff bound, the probability that any vertex has fewer than
  $\eta\delta|G|/4$ neighbours in $P$ is at most
  $|G|\exp(-\eta\delta|G|/16)$, which is less than $1/2$ if $D_4$ is large. 
  Therefore with positive probability there is a set $P$ with $|P|\le
  2\eta|G|$ such that every vertex not in $P$ has at least
  $\eta\delta|G|/4\ge 2\eta^2|G|$ neighbours in $P$. Make such a choice of
  $P$: it satisfies properties~(1) and~(3).

  Finally, consider the bipartite graph with $A=X$ and
  $B=P-Y$. By choice of~$P$, each vertex of $A$ has at least
  $\eta\delta|G|/4-|Y|\geq \eta\delta|G|/8$ neighbours in $B$, which is at
  least $\delta|B|/16$ by property~(1). Lemma~\ref{L:bipartiteprojection}
  then implies there is a set $M$ in $P-Y$ with $X\subset \Gamma(M)$ and
  $|M|\leq\floor{\log_{1/(1-\delta/16)}\abs{X}}+1
  \leq\floor{\log_{1/(1-\eta/2)}\abs{X}}+1$. Thus property~(4) holds.
\end{proof}

The next theorem shows that the connector and projector do the job required
of them in the discussion at the start of~\S\ref{secdense}. The theorem
accommodates general partitions of a set; in this paper we shall apply it
only to equipartitions, but the more general result will be useful
elsewhere (see~\S\ref{secfuture}).

\begin{theorem}
\label{T:connectpartition}
Given $\eta>0$ there is a constant $D_5(\eta)$ with the following
property. Let $G$ be a graph and $R\subset V(G)$ satisfy the conditions of
Theorem~\ref{T:connectorprojector}.  Let $C$ and $P$ be sets given by that
theorem. Suppose further that $|G|\geq D_5|R|$. Then for any partition
$(V_r: r\in R)$ of $G-C-P-R$ into $|R|$ parts, together with a set
$F\subset\{\,rs\,:\,r,s\in R\}$ of pairs from $R$ with $|F|\le |R|/\eta$,
there are disjoint sets $U_r$, $r\in R$ such that
\begin{enumerate}
\item{}$V_r\cup \{r\}\subset U_r$ for all $r\in R$, 
\item{}$G[U_r]$ is connected for all $r\in R$, and
\item{} there is a $U_r$-$U_s$ edge for every pair $rs\in F$.
\end{enumerate}
\end{theorem}
\begin{proof}
  Note that the conditions of Theorem~\ref{T:connectorprojector} imply
  $\eta\leq1/8$. Our first aim is to choose, one by one for each $r\in R$,
  disjoint sets $M_r\subset P$ with $V_r\cup\{r\}\subset \Gamma(M_r)$:
  for later parts of the proof, we shall require that
  $\sum_{r\in R}|M_r|\le \eta^{(3+1/2\eta)}|G|$. To this end, 
  let $C_\eta$ be a constant, depending on~$\eta$, such that
  $1+\log_{1/(1-\eta/2)}(x+1)\le C_\eta\log(2x)$ holds for all $x\ge 1$.
  The sets $M_r$ will be derived via Theorem~\ref{T:connectorprojector} and
  satisfy $|M_r|\le \floor{\log_{1/(1-\eta/2)}\abs{V_r\cup\{r\}}}+1$; note
  that therefore $|M_r|\le C_\eta\log2|V_r|$. Thus $\sum_{r\in R}|M_r|\leq
  C_\eta\sum_{r\in R}\log2|V_r| \leq C_\eta|R|\log (2|G|/|R|)$ holds
  by the concavity of the $\log$ function. Now $\log(2x)/x\to0$ as
  $x\to\infty$ so $D_5$ can be chosen so that
  $ C_\eta (|R|/|G|)\log (2|G|/|R|) \leq \eta^{(3+1/2\eta)}$, which
  means that
  $\sum_{r\in R}|M_r|\le \eta^{(3+1/2\eta)}|G|$. We now see that our first
  aim can indeed be realised: choose the $M_r$ one by one, each time
  applying property~(4) of Theorem~\ref{T:connectorprojector} with
  $X=V_r\cup\{r\}$ and $Y$ being the union of those $M_{r'}$
  already chosen, noting that  $|Y| \leq \sum_{r\in R}|M_r|\leq
  \eta^{(3+1/2\eta)}|G|\leq \eta^2|G|$.

  We now choose disjoint sets $P_r\subset P$, $r\in R$, such that
  $G[P_r\cup M_r]$ is connected, and therefore so also is
  $G[P_r\cup M_r \cup V_r\cup\{r\}]$. To do this, we find $|M_r|-1$ paths
  joining the vertices of $M_r$, whose internal vertices are in~$C$, making
  use of property~(2) of Theorem~\ref{T:connectorprojector}. We choose the
  paths one by one. When we come to join a pair $u$, $v$ of vertices, some
  of the $2\eta^{2+1/2\eta}|G|$ paths given by property~(3) will be
  unavailable, because they contain a vertex lying in some previously
  chosen path. But at most $\sum_{r\in R}|M_r|\le \eta^{(3+1/2\eta)}|G|$
  paths were previously chosen, so at most
  $(1/2\eta) \eta^{(3+1/2\eta)}|G| < \eta^{2+1/2\eta}|G|$ $u$--$v$ paths
  are unavailable. Hence the desired sets $P_r$ can all be found.

  We now take $U_r=P_r\cup M_r \cup V_r\cup\{r\}$. This gives
  properties~(1) and~(2) of the theorem. To obtain property~(3), for each
  $rs\in F$ we find an $r$--$s$ path $Q_{rs}$ whose internal vertices lie
  in~$C$. These can be found one by one by an argument similar to that just
  given; the total number of unavailable $r$--$s$ paths, accounting for the
  sets $P_r$ as in the previous paragraph and for the paths $Q_{r's'}$
  already chosen, is at most
  $\eta^{2+1/2\eta}|G| + (2/\eta)(|R|/\eta) < 2\eta^{2+1/2\eta}|G|$ if
  $D_5$ is large, so
  at least one path is available.  Having found $Q_{rs}$, add the vertices
  of $Q_{rs}$, apart from~$r$, to $U_s$. In this way we arrive at sets
  satisfying properties~(1)--(3).
\end{proof}

\begin{dfn} Let $H$ and $G$ be graphs, and let $R\subset V(G)$  be a set of
  $|H|$ vertices labelled by the vertices of $H$; say $R=\{\,r_v \,:
  \, v\in V(H)\}$. We say that $G$ has an $H$ minor {\em
    rooted at} $R$ if there exist non-empty disjoint subsets
  $U_v: v\in V(H)$ of $V(G)$ with $r_v\in U_v$, such that each $G[U_v]$ is
  connected and, whenever $vw$ is an edge in $H$, there is an edge in $G$
  between $U_v$ and $U_w$.
\end{dfn}

We are finally ready to state and prove the main theorem of this
section. This is a strengthening of Theorem~\ref{T:sectwoweak} that
gives a rooted minor, which, as we mentioned earlier, will be useful later
on. Note that Theorem~\ref{T:sectwoweak} follows from
Theorem~\ref{T:sectwostrong} by picking an arbitrary set $R$ of roots.

\begin{theorem}
\label{T:sectwostrong}
Given $\epsilon >0$, there exists $D_1(\epsilon)$, such that
if $H$ is a graph of average degree $d>D_1$, $G$ is a graph of density at
least $p+\epsilon$, with the properties $\epsilon < p <1-\epsilon$,
$|G|\geq |H|\sqrt{\log_{1/(1-p)}d}$ and $\kappa(G)\geq \epsilon |G|$, and
furthermore $R\subset V(G)$ is a set of vertices labelled by
$V(H)$, then $G$ has an $H$ minor rooted at~$R$.
\end{theorem}
\begin{proof}
  Note that we can assume throughout that both $|G|$ and $|G|/|H|=|G|/|R|$
  are arbitrarily large, since $\log(1/(1-p))$ is bounded above by the
  constraint on~$p$. We begin by applying
  Theorem~\ref{T:connectorprojector} with $\eta = \epsilon/20$ to obtain
  sets $C$ and $P$ as in the theorem. Consider the graph $G'=G-C-P-R$. We
  want to apply Theorem~\ref{C:almostcompatible} to~$G'$, but $|G'|$ might
  not satisfy the stated lower bound. However, we may assume that
  $|G'|\geq (1-4\eta)|G|-|R|\geq(1-\epsilon/4)|G|$, so
  $e(G')\geq e(G)-(\epsilon/4)|G|(|G|-1)$. Thus
  $e(G')\geq (p+\epsilon/2){|G|\choose2}$, and $G'$ has density at least
  $p+\epsilon/2$. Let $p'=p+\epsilon/4$ and let $\epsilon'=\epsilon/4$. By
  Lemma~\ref{L:logprop},
  $\sqrt{\log_{1/(1-p')}d}\leq(1-\epsilon/4)\sqrt{\log_{1/(1-p)}d}$. So we
  may apply Theorem~\ref{C:almostcompatible} to $G'$ with parameters $p'$
  and $\epsilon'$ to obtain an $|H|d^{-\epsilon/12}$-almost-$H$-compatible
  equipartition $V_v$, $v\in V(H)$ of $G'=G-C-P-R$.  Apply
  Theorem~\ref{T:connectpartition} to this equipartition (formally writing
  $V_{r_v}$ instead of $V_v$), with $F$ the set of pairs $r_vr_w$ for which
  there is no $V_v$--$V_w$ edge; note that
  $|F|\leq |H|d^{-\epsilon/12}\leq|H|<|R|/\eta$ if $d$ is large. The
  resulting sets $U_r$ then form an $H$ minor in $G$ rooted at $R$.
\end{proof}

\section{The Sparse Case}\label{secsparse}
Our approach in the sparse case broadly mirrors that of~\cite{thom01},
except that we need to construct an $H$ minor rather than a complete minor,
and the graphs $G$ in which we are working have many fewer edges. (Both
these difficulties were sidestepped in~\cite{myersthom}, where there were
enough edges in the sparse case to find a large complete minor.)  We will
construct a large number (depending on~$\epsilon$) of disjoint dense and highly
connected subgraphs of a minor-minimal graph --- in all but one of these we
find different subgraphs of $H$ as rooted minors, using
Theorem~\ref{T:sectwostrong}, and then we connect these minors together
using the remaining dense subgraph.

To find the dense subgraphs, we use the fact that, in a minor minimal
graph, a typical vertex has a dense neighbourhood. Either we can find a
large number of typical vertices whose neighbourhoods are largely disjoint,
in which case we can carry out the programme just described, or most
vertices have highly overlapping neighbourhoods. The latter case will be
handled by the next lemma. This lemma is much the same as
\cite[Lemma~5.1]{thom01}, but we need to reprove it because the degrees in
our graphs are much lower and the constants involved are different.

\begin{lemma}
  \label{T:largebipminor}
  Given $\epsilon, C>0$, there exists $D_6(\epsilon, C)$ such that, if
  $d>D_6$, the following holds.  Let $H$ be a graph of average degree~$d$,
  and let $n\geq \epsilon |H|\sqrt{\log d}$. Let $G$ be a bipartite graph
  with vertex classes $A$, $B$ such that $\abs{B}\leq Cn$,
  $\abs{A}\geq D_6n$, and every vertex of $A$ has at least $n$ neighbours
  in~$B$. Then $G\succ H$.
\end{lemma}
\begin{proof}
  Begin by choosing a number $0<q<1/5$ such that
  $\epsilon \sqrt{-\log(2q)}\geq 1$. This implies that
  $n\geq|H|\sqrt{\log_{1/2q}d}$. Then take $D_6$ so that
  $D_6q\ge C^2$.

  Proceed by contracting, one by one, each vertex $a\in A$ to a vertex
  $b\in B$ of minimal degree in $G^*[N(a)]$, where $G^*$ is the graph we
  have at the moment we contract $ab$. We imagine that $a$ disappears in
  this process but~$b$ remains, and after we have done this for all
  $a\in A$ we are left with a graph on vertex set~$B$. Note that, if
  $G^*[N(a)]$ has minimum degree $p_a(|N(a)|-1)$, and $q_a=1-p_a$, then we
  add at least $q_a(n-1)$ edges to~$B$. Suppose that $q_a\ge q$ for every
  $a\in A$. Then we have added at least $|A|q(n-1)$ edges. But we can add
  at most ${|B|\choose2}$, so $2|A|q(n-1)\leq |B|(|B|-1)$. However, this
  inequality fails by our choice of $D_6$.

  It follows that, for some $a\in A$, $q_a\le q$ holds. Let
  $G'=G^*[N(a)]$. Then $G'$ has minimum degree at least $p_a(|G'|-1)$,
  and since $q_a\leq 1/5$ we have $p_a\geq 4/5$ so $\kappa(G')\geq |G'|/2$.
  Moreover $|G'|\geq n\geq |H|\sqrt{\log_{1/2q}d}$.  Now $G'$ has density
  at least $(1-2q)+q$, so we can apply Theorem~\ref{T:sectwoweak} to $G'$
  with parameters $\epsilon=q$ and $p=1-2q$ (increasing $D_6$ if necessary
  so that $D_6\geq D_1(q)$), to obtain $G\succ G'\succ H$.
\end{proof}

We can now prove Theorem \ref{T:secthree}. We restate first it, for convenience.

\secthreethm*

\begin{proof}
   Let $\ell=\ceil{8/\epsilon}$ and $L={\ell\choose2}$.
   We start by finding, one by one, disjoint sets $S_0,\ldots,S_L$ such that
   $$
   \abs{S_i}\leq 6\enklet{} \quad \mbox{and}\quad
   \delta(G[S_i])> \frac{5}{6}\enklet{}-1\,, \quad\mbox{for } 0\le i\le L\,.
   $$
   Suppose we have already found sets $S_0,...,S_{k-1}$, where $k\leq L$.
   Let $B = \bigcup_{0\leq i < k}S_i$. Then $\abs{B}\leq 6L\enklet{}$. Let
   $A$ be the set of vertices in $V(G)-B$ having degree at
   most~$6\enklet{}$.  Then
   $3\enklet{}\abs{G-B-A}\leq e(G)\leq \enklet{}\abs{G}$.  In particular,
   $\abs{G-B-A}\leq \abs{G}/2$. Assuming, as we may, that $D_2\geq 36L$,
   then $|B|\leq |G|/6$, so $|A|\geq |G|/3$.  Let $n=\enklet{}/6$ and let
   $C=36L$, so $n\geq (\epsilon/6)|H|\sqrt{\log d}$ and $|B|\le Cn$.  We
   may also assume that $D_2\ge D_6(\epsilon/6,C)/2$, where $D_6$ is the
   constant in Lemma~\ref{T:largebipminor}, and so $|A|>D_6n$. Then, by
   that lemma, either $G\succ H$, in which case the theorem is proved and
   we are done, or some vertex $a\in A$ has fewer than $n=\enklet{}/6$
   neighbours in~$B$. In this case put $S_k=\Gamma(a)\setminus B$.  Then
   $|S_k|\leq 6\enklet{}$, since $a\in A$. Moreover, because each edge
   incident with $a$ lies in more than $\enklet-1$ triangles, we have
   $\delta(G[\Gamma(a)])\geq m-1$ and so $\delta(G[S_i])>
   5\enklet{}/6-1$. We thus find all our sets $S_0,\ldots,S_L$.

   Next, inside each $S_i$ we find a subset $T_i$ with
   $\kappa(G[T_i])\geq \enklet{}/40$ and \\$\delta(G[T_i])\geq3\enklet{}/4$.
   If $\kappa(G[S_i])\geq\enklet{}/40$, just put $T_i=S_i$. If not, remove a
   cutset of size at most $\enklet{}/40$ from $G[S_i]$ and let $S_i'$ be the
   set of vertices of a smallest component. Then $|S_i'|\leq|S_i|/2\le
   3\enklet$ and
   $\delta(G[S_i'])\geq 5\enklet{}/6-1-\enklet{}/40$. If
   $\kappa(G[S_i'])\geq \enklet{}/40$, put $T_i=S_i'$. Otherwise,
   repeat the procedure on $G[S_i']$. After $j$ repetitions of the procedure
   we have a subgraph with at most $6\enklet/2^j$ vertices and minimum
   degree at least $5\enklet{}/6-1-j\enklet{}/40$. This is impossible for
   $j=4$, so we reach the desired $T_i$ in at most 3~steps.

   Now let $V_1,...,V_\ell$ be an arbitrary equipartition of $H$, and let
   $H^{\{i,j\}} = H[V_i\cup V_j]$. Each subgraph $H^{\{i,j\}}$ has at most
   $2|H|/\ell+2\leq \epsilon|H|/4+2$ vertices, and average degree at most
   $\ell d$ (if $D_2$ is large). Notice that $H$ is the (not edge-disjoint)
   union of the $L = \binom{\ell}{2}$ subgraphs $H^{\{i,j\}}$.  We shall find the
   $H^{\{i,j\}}$ as minors inside the $G[T_i]$, and with this in mind we
   relabel $T_1,\ldots,T_L$ as $T^{\{i,j\}}:1\leq i<j\leq \ell$; we shall
   find $H^{\{i,j\}}$ in $G[T^{\{i,j\}}]$.

   We now describe how the $H$ minor will be formed, leaving the details of
   the construction to later. The minor will be rooted at a set of roots
   $R\subset T_0$, which we pick now and label as
   $R=\{r_v:v\in V(H)\}$. The $H^{\{i,j\}}$ minors also need to be rooted,
   so pick now a set of roots
   $R^{\{i,j\}}=\{r^{\{i,j\}}_v: v\in V_i\cup V_j\}\subset T^{\{i,j\}}$. For
   each pair $\{i,j\}$ and for each $v\in V_i\cup V_j$, we find an
   $r_v$--$r^{\{i,j\}}_v$ path $P^{\{i,j\}}_v$, such that all the paths
   $P^{\{i,j\}}_v$ are internally disjoint from each other and from all the
   $H^{\{i',j'\}}$ minors. Then the paths $P^{\{i,j\}}_v$ with the minors
   $H^{\{i,j\}}$ together give an $H$ minor, because every edge of $H$ lies
   in one of the $H^{\{i,j\}}$, and by contracting the
   paths $P^{\{i,j\}}_v$ we identify, for each $v\in V(H)$, all the root
   vertices $r^{\{i,j\}}_v$ that are labelled by~$v$.

   Here are the constructional details. In practice, to avoid the paths
   $P^{\{i,j\}}_v$ intersecting the minors $H^{\{i',j'\}}$, we construct the
   paths first, then remove them and find the minors in the remaining
   graph. Each vertex $v$ lies in $\ell-1$ sets $V_i\cup V_j$, so
   $R^*=\bigcup_{\{i,j\}} R^{\{i,j\}}$ satisfies
   $|R^*|=(\ell-1)|H|$. Note that, if $D_2$ is large, then
   $\kappa(G-R)\geq D_2|H|-|R^*|\geq|R^*|$ and
   $|T_0-R|\geq 3m/4-|H|\geq|R^*|$. Thus by Menger's theorem we can find $|R^*|$
   vertex disjoint paths in $G-R$ joining the set $T_0-R$ to the
   set~$R^*$. Let the path which ends at $r^{\{i,j\}}_v$ be
   $Q^{\{i,j\}}_v$, and let its first vertex be $x^{\{i,j\}}_v\in T_0$. 
   These paths $Q^{\{i,j\}}_v$ might, as they stand, use a lot of vertices
   from the sets $T^{\{i',j'\}}$, so we now modify them so as to avoid this.

   Recall that $|T_0|\le 6\enklet$ and $\kappa(G[T_0])\geq \enklet{}/40$.
   Applying Lemma~\ref{L:connectivity} to $G[T_0]$, we see that each pair
   of vertices is joined by at least $\enklet{}/38400$ paths of length at
   most~$480$ --- we call these ``short'' paths. Exactly the same remark
   applies to each $G[T^{\{i,j\}}]$. We now modify the paths
   $Q^{\{i,j\}}_v$, one by one, in the following way. For each set
   $T^{\{i',j'\}}$ that the path $Q^{\{i,j\}}_v$ enters, let $x$ and $y$ be
   the first and last vertices of $Q^{\{i,j\}}_v$ in $T^{\{i',j'\}}$, and
   replace the section of $Q^{\{i,j\}}_v$ between $x$ and $y$ by a short
   path inside $G[T^{\{i',j'\}}]$. At the moment we don't require these
   short paths to be disjoint for different $i,j,v$, but after doing this
   the paths $Q^{\{i,j\}}_v$ have the property that they enter each
   $T^{\{i',j'\}}$ only once, and use at most $480$ vertices of
   $T^{\{i',j'\}}$. After this is done, note that the first and last
   vertices of the $Q^{\{i,j\}}_v$ in the $T^{\{i',j'\}}$ are all still
   distinct, since they can never lie on the internal vertices of our new
   short paths - they are always vertices of the original path.  Let $E$ be
   the set of first and last vertices of all the new short paths in all the
   $T^{\{i,j\}}$.

   We now go through each $T^{\{i,j\}}$ and replace the short paths used by
   ones that are disjoint from each other and from~$E$. We do this by
   replacing the short paths one at a time, each time choosing a new short
   path disjoint from previously chosen new short paths and disjoint
   from all the endpoints~$E$. We can do this because
   the number of vertices we need to avoid is at most
   $480|R^*| +|E|\leq 482|R^*| < \enklet{}/40000$. The modified paths
   $Q^{\{i,j\}}_v$ are now once again vertex 
   disjoint, each using at most $480$ vertices from each
   $T^{\{i',j'\}}$. Finally, we extend the paths $Q^{\{i,j\}}_v$ to the
   paths $P^{\{i,j\}}_v$ that we want, by finding $|R^*|$ internally
   disjoint short paths inside $G[T_0]$ that join $r_v$ to $x^{\{i,j\}}_v$,
   for all $v\in V(H)$ and all appropriate~$\{i,j\}$.

   What remains is to find the $H^{\{i,j\}}$ minors. Fix some
   pair~$\{i,j\}$. Let $X$ be the set of vertices on the paths
   $P^{\{i',j'\}}_{v'}$ that lie inside~$T^{\{i,j\}}$, other than the roots
   $R^{\{i,j\}}$. By the choice of these paths, we have
   $|X|\le 480|R^*|< 480\ell|H|$. Let $G'=G[T^{\{i,j\}}-X]$. Then it is
   enough to find an $H^{\{i,j\}}$ minor in $G'$ that is rooted at
   $R^{\{i,j\}}$. Recall that $H^{\{i,j\}}$ has at most $\epsilon|H|/4+2$
   vertices, and average degree at most $\ell d$. Since $H$ itself has
   average degree~$d$, it is possible to add edges to $H^{\{i,j\}}$, if
   necessary, so that the resultant graph $H'$ has average degree $d'$
   where $\epsilon d/5\leq d'\leq \ell d$ (if $D_2$ is large). It is now
   enough to find $H'$ as a minor of $G'$, rooted at $R^{\{i,j\}}$.

   Recalling the properties of the~$T_i$, we see that $|G'|\leq 6m$,
   $\delta(G')\geq 3\enklet{}/4-|X|$ and $\kappa(G')\geq \enklet{}/40-|X|$.
   If $D_2$ is large, this means $\delta(G')> 2\enklet{}/3\geq |G'|/9$ and
   $\kappa(G')\geq \enklet{}/50\geq |G'|/300$. Let $\epsilon'=1/300$.  Then
   we can pick~$p\ge 1/10$ such that $\epsilon'<p<1-\epsilon'$ and $G'$ has
   density at least $p+\epsilon'$. Moreover, $\delta(G')\geq 2\enklet{}/3$
   implies we can choose $p+\epsilon' \geq 2\enklet{}/3|G'|$, and since
   $\epsilon'\leq m/50|G'|$ we have $p\geq 3\enklet{}/5|G'|$. Recalling that
   $|H'|\leq \epsilon|H|/4+2$, and that $\epsilon d/5\leq d'\leq \ell d$, we
   obtain, if $D_2$ is large,
   $$
   |G'| \geq \frac{3\enklet{}}{5p}\geq \frac{3\epsilon}{5p}
   |H|\sqrt{\log d}\geq \frac{11}{5p}|H'|\sqrt{\log d}
   \geq \frac{2}{p}|H'|\sqrt{\log d'}\,.
   $$
   By the definition of~$\alpha>1/4$, this implies
   $$
   |G'| \geq 4\frac{\sqrt{\log(1/(1-p))}}{p/2}
   |H'|\sqrt{\log_{1/(1-p)} d'} > |H'|\sqrt{\log_{1/(1-p)} d'}\,.
   $$
   Because $d'>\epsilon d/5$ we can ensure that $d'>D_1(\epsilon')$ by
   making $D_2$ large. We now have all the conditions we need to conclude,
   from Theorem~\ref{T:sectwostrong}, that $H'$ is a minor of $G'$ rooted
   at $R^{\{i,j\}}$, as required.
\end{proof}

\section{Further extensions}\label{secfuture}

As mentioned in the introduction, Myers and Thomason~\cite{myersthom}
defined a graph parameter $\gamma(H)$ and proved that
$c(H)=(\alpha\gamma(H)+o(1)) t\sqrt{\log t}$ for graphs $H$ with $t$
vertices and at least $t^{1+\tau}$ edges. The parameter $\gamma(H)$ is
found by considering non-negative vertex weightings $w:V(H)\to{\mathbf
  R}^+$, and is given by
$$
\gamma(H)\,=\,\min_w\,{1\over t}\,\sum_{u\in H}w(u)
\mbox{\qquad such that\qquad}
\sum_{uv\in E(H)}\,t^{-w(u)w(v)}\,\le\,t\,.
$$
The constant weighting $w(v)=\sqrt \tau$ satisfies the constraints, so
$\gamma(H)\le\sqrt\tau$ always, but in general the relative sizes of the vertex
weights are, essentially, the relative sizes of the sets $|U_v|$ that are most
likely to give an $H$ minor in a random graph (this is where the parameter
comes from). For most $H$, and for regular $H$ when $\tau>0$, the optimal
sizes of the $|U_v|$ are the same, so $\gamma(H)\approx\sqrt\tau$ and
$c(H)=(\alpha+o(1)) |H|\sqrt{\log d}$.  But there are natural examples
where taking equal sized sets $|U_v|$ is not optimal. For example, the
complete bipartite graph $K_{\beta t, (1-\beta)t}$ satisfies
$\gamma(K_{\beta t, (1-\beta)t})\approx 2\sqrt{\beta(1-\beta)}$.

When $H$ is sparse, examples such as the hypercube cited in the
introduction suggest that $\gamma(H)$ is not enough on its own to determine
$c(H)$. Nevertheless we think that Theorem~\ref{T:key} can be extended to
incorporate some features of $\gamma(H)$. For example, if $\beta$ is fixed,
the proof of Theorem~\ref{T:key} could probably be modified to show that,
if $d$ is large and $H$ is a bipartite graph with average degree~$d$,
having vertex class sizes $\beta|H|$ and $(1-\beta)|H|$, then
$c(H)\le (2\sqrt{\beta(1-\beta)}\alpha+o(1))|H|\sqrt{\log d}$.  Moreover,
the argument of~\cite{NRTW} indicates that equality holds for almost all
such~$H$. Similar remarks could be made regarding multipartite graphs.
With this in mind, we have, as mentioned earlier, stated
Theorem~\ref{T:connectpartition} in a form more general than what is needed
in this paper, but other than that we do not pursue any details of these
ideas here.

\bibliography{extremalbib.bib}
	
\bibliographystyle{plain}

\end{document}